\documentclass{article}
\usepackage[utf8]{inputenc}
\usepackage{listings,amsmath,amssymb,xcolor,graphicx,url}
\usepackage{amsthm}
\usepackage{amssymb}
\usepackage{algorithm}
\usepackage[noend]{algpseudocode}
\renewcommand{\t}{\tilde{t}}
\usepackage{comment}
\usepackage[pagewise]{lineno}
\usepackage{authblk}

\newcommand{\R}{\mathbb{R}}

\newcommand{\x}{\boldsymbol{x}}

\newtheorem{theorem}{Theorem}
\newtheorem{lemma}{Lemma}

\newtheorem{remark}{Remark}
\newtheorem{defn}{Definition}

\title{Simultaneous recovery of attenuation and source density in SPECT}
\author{Sean Holman and Philip Richardson\\
Department of Mathematics, University of Manchester,\\
Oxford Road, Manchester M139PL, United Kingdom\\
Corresponding author email: {\it sean.holman@manchester.ac.uk}}
\date{\today}

\begin{document}
\maketitle
\begin{abstract}
We show that under a certain non-cancellation condition the attenuated Radon transform uniquely determines piecewise constant attenuation $a$ and piecewise $C^2$ source density $f$ with jumps over real analytic boundaries possibly having corners. We also look at numerical examples in which the non-cancellation condition fails and show that unique reconstruction of multi-bang $a$ and $f$ is still appears to be possible although not yet explained by theoretical results.
\end{abstract}

\section{Introduction}
This paper considers the problem of Single Photon Emission Computed Tomography (SPECT) in which we seek to recover both attenuation $a$ and radiation source density $f$. Given an attenuation $a$ and source density $f$, which are both functions on $\mathbb{R}^2$, the attenuated X-ray transform is
\begin{equation}\label{eq,AtRT}
R_af(\x,\theta) = \int_{-\infty}^\infty f(\x+t\theta) e^{-Da(\x+t\theta,\theta)}\ \mathrm{d} t
\end{equation}
where $\x\in \mathbb{R}^2$ and $\theta \in \mathbb{S}^1$, with $Da$ the beam transform given by 
\begin{equation}\label{eq,Btrans}
Da(\x,\theta)=\int_{0}^{\infty}a(\x+\rho\theta)\mathrm{d}\rho.
\end{equation}
Note that we are considering only the two dimensional case. For compactly supported $a$ and $f$ on some connected Lipschitz bounded domain  $\Omega\subset\mathbb{R}^2$, the attenuated X-ray or Attenuated Radon Transform (AtRT) occurs naturally in the solution of the 2D photon transport equation \cite{atrt}
\begin{equation}\label{phtranseq}
\begin{split}
\theta\cdot\nabla{u(\x,\theta)}+a(\x)u(\x,\theta) = f(\x), \quad (\x,\theta)\in \Omega\times\mathbb{S}^{1},\\
u|_{\Gamma^{-}}=0,
\end{split}
\end{equation}
where $u(\x,\theta)$ is the photon flux through the point $
\x$ in a unit direction $\theta\in\mathbb{S}^{1}$ and 
$$ \Gamma^{-}=\{(\x,\theta)\in\partial\Omega\times\mathbb{S}^{1}\ |\ \theta\cdot \underline{n}(\x)\leq 0\}$$
with $\underline{n}(\x)$ the unit outward pointing normal to the boundary at $\x$. Note that for dimensions larger than two the (attenuated) Radon Transform and (attenuated) X-ray transform are not the same \cite{Natterref}, but in dimension two they coincide up to parametrisation of lines and because of this we use the terms interchangeably.

The problem we will consider is determination of both $f$ and $a$ from $R_af$ which is not possible in general. This problem of recovering both $a$ and $f$ from the AtRT is sometimes called the SPECT identification problem {\cite{Natterref,Ident,Identification1,Holrich}}.  The authors of {\cite{GourionNoll}} make use of the ideas in \cite{Identification1} to show non-uniqueness for radial $a$ and $f$. This means that there are pairs of $(a_1,f_1)$ and $(a_2,f_2)$ which depend only on distance to the origin which satisfy $R_{a_1}f_1=R_{a_2}f_2$. Furthermore, numerical evidence given in \cite{Phil} shows that recovery is unstable for pairs of $a$ and $f$ which are close to being radial. A more obvious case where unique recovery is not possible is when $f=0$, as any choice of $a$ will trivially give $R_{a}0=0$.

When $a$ is known, the mapping $f \mapsto R_a f$ has been shown to be invertible under certain mild conditions on decay at infinity. In the more general setting, similar formulae have been established for attenuated tensor transforms \cite{Monard1}, and used to study the problem on surfaces for both functions and tensor fields \cite{HMS,VPKrish}. When these conditions hold, a closed form solution for the inverse is known, see \cite{inversion,Identification2,Identification3}. This means that if it is possible to recover $a$ uniquely via some method then we obtain $f$ for free.

Although unique recovery of $a$ and $f$ is impossible in the general case, under additional hypotheses it can be possible to determine $a$ and $f$. A large amount of work has been carried out in the field of medical imaging which focuses on numerical methods for the SPECT identification problem (see \cite{attcorr1,attcorr2,attcorr3} and their references), although the standard practical approach for SPECT is to first determine $a$ through a separate Computerized Tomography (CT) scan. The author of \cite{Bronnikov1} uses linearization on SPECT data to try to determine $a$ alone. The same author also gives range conditions in this case \cite{Bronnikov2,Bronnikov3}. Another approach is to make use of scattered photons for additional information for the SPECT identification problem \cite{scatter1,scatter2} but the forward model needs to be modified to include scattering which changes the mathematical problem. If $Da$ in \eqref{eq,AtRT} is replaced by a constant $\mu$ times $t$ the transform is called the exponential Radon transform. In this case, it is shown that $\mu$ can be determined from the exponential Radon transform exactly when $f$ is not radial \cite{Identification1}. {Note that $f$ is not assumed to be known for this result}. 

Most recently, work in \cite{Holrich} makes use of a multi-bang assumption on $a$. This work is inspired by the convex multi-bang regularization technique given in \cite{MB,MBorig} designed to reconstruct images in which the values expected in the image are already known. The authors of \cite{Holrich} show unique recovery of $a$ and $f$ when $f$ is $C^2_{c}(\Omega)$ and $a$ is piecewise constant over a series of nested convex sets, called nicely multi-bang. This paper extends the results given in \cite{Holrich} to consider the case where $f$ is allowed to have certain discontinuities which potentially coincide with some of the discontinuities of $a$ and allow $a$ to be piecewise constant (i.e. to remove the ``nicely multi-bang" hypothesis). In this work we have two essential hypotheses about $a$ and $f$:
\begin{enumerate}
    \item We assume that $a$ is piecewise constant with discontinuities occurring across analytic curves possibly with corners and $f$ is piecewise $C^2$ with discontinuities also occurring across analytic curves possibly with corners.
    \item We assume that discontinuities satisfy a certain non-cancellation condition given by \eqref{eq,edgecancel2}.
\end{enumerate}
We will make these assumptions more precise in Section \ref{sec,thmdef} but note that while the first is realistic, the second implies that if $a$ has a discontinuity at a location where $f$ does not, then the support of $f$ must extend outside the support of $a$ which is not realistic for most applications (i.e. it requires at least some of the radiative source to be outside the object). On the other hand, if $a$ and $f$ are always discontinuous at the same places, then this is not required.

{The rest of the paper is structured as follows. Section \ref{sec,theory} introduces some necessary definitions in subsection \ref{sec,theory:prelim} as well as notation used in the paper given in \ref{sec,theory,notation}. Section \ref{sec,thmdef} states and proves the main result, Theorem \ref{thm1}, as well as some related results in subsections \ref{subsec,jumps} and \ref{subsec,boundarydetermine}.  Section \ref{sec,NumericalMethod} describes the numerical methods used and section \ref{sec,NumericalExamples} gives some numerical examples. The final section concludes the work in the paper and suggests avenues for further research.}

\section{Definitions and Notation}\label{sec,theory}


\subsection{Preliminaries}\label{sec,theory:prelim}
This subsection contains technical definitions necessary for the statement of our main results.{ We begin by defining the types of regions over which $a$ is constant.} Recall that a function is analytic if it is $C^\infty$ and equal to its Taylor series in a neighbourhood of every point.
	\begin{defn}{\bf{(Analytic boundary)}}\label{def:analyticboundary}
		Suppose $\Omega$ is a region in $\mathbb{R}^2$. For a point $\x^* \in \partial \Omega$ we say that $\partial \Omega$ is analytic near $\x^*$ if there exists a neighbourhood V of $\x^*$ and a set of Cartesian coordinates $(x,y)$ {centred at $\x^{*}$} such that on $V$ the boundary is given by $y=f(x)$ where $f$ is an analytic function.
	\end{defn}
	\begin{defn}{\bf{(Analytic corner)}}\label{def:corner}
		Let $\x^*$ be a point on the boundary of a region $\Omega$. Then $\x^*$ is an analytic corner point of $\Omega$ if there exists a neighbourhood $V$ of $\x^*$ and Cartesian coordinates $(x,y)$ {centred at $\x^{*}$} such that we can describe the boundary {for $x<0$} via $y=f(x)$ and {for $x>0$} via $y=g(x)$, where both $f$ and $g$ are analytic in a neighbourbood of $x = 0$.
	\end{defn}
	\begin{defn}{\bf{(Piecewise analytic boundary with corners)}}\label{def:piecewisewithcorners}
		A set $\Omega$ has piecewise analytic boundaries with corners if for every point $\x^*$ in $\partial \Omega$, $\x^*$ is either an analytic corner point for $\Omega$, or $\partial \Omega$ is analytic near $\x^*$.
	\end{defn}
	With these definitions we are now ready to define precisely what we mean by a multi-bang function.
	\begin{defn}{\bf (Multi-bang)} \label{def:multibang}
		We say that $a \in L^\infty(\mathbb{R}^2)$ is multi-bang with admissible set $\mathcal{A} \subset \mathbb{R}$ if $\mathcal{A}$ is a finite set and there exists a collection of disjoint bounded open sets $\{\Omega_{j}\}_{j = 1}^n$ with piecewise analytic boundaries possibly having corners such that
		\begin{equation} \label{dmultibang}
		a = \sum_{j=1}^n a_j \chi_{\Omega_j}
		\end{equation}
		where $a_j \in \mathcal{A}$ for all $j$. Here $\chi_{\Omega_j}$ is the characteristic function of the set $\Omega_j$, and we assume that for all $\Omega_j$ the interior of the closure of $\Omega_j$ is equal to $\Omega_j$.
	\end{defn}
	 
	 We also precisely define what we mean by a piecewise $C^2$ function with piecewise analytic boundary with corners. We will denote such functions as p.a.b. $C^2$.
	\begin{defn}{\bf (Piecewise $C^2$ with piecewise analytic boundaries with corners (p.a.b. $C^2$))} \label{def,piecewiseanalyticf} 
We say that $f \in L^\infty(\mathbb{R}^2)$ is piecewise $C^2$ with piecewise analytic boundaries with corners (p.a.b. $C^2$)
	    if there exists a collection of sets  $\{\Omega_{j}\}_{j = 1}^n$ satisfying the same hypotheses as in Definition \ref{def:multibang} such that
		\begin{equation} \label{eq,fpiecewiseC2}
		f = \sum_{j=1}^n f_j \chi_{\Omega_j}
		\end{equation}
	where each $f_{j}\in C^{2}(\mathbb{R}^2)$.
	\end{defn}
		Definitions \ref{def:piecewisewithcorners}, \ref{def:multibang} and \ref{def,piecewiseanalyticf} allow for a lot of different choices for $a$ and $f$ but eliminate pathological cases such as highly oscillatory behaviour which could occur if we only require smooth boundaries. Note that when $f\in C^{\infty}_{c}(\mathbb{R}^2)$, so that $f$ has no discontinuities, and the boundaries of $a$ correspond to nested convex sets we recover the nicely multi-bang setup examined in \cite{Holrich}.
		
\subsection{Notation}\label{sec,theory,notation}

This section includes some notation that we will use throughout the paper.
We consider the attenuated X-ray transform along directed lines of the form
\[
L(\x,\theta(\omega)): t \mapsto \x + t\theta(\omega)
\]
where $\x \in \mathbb{R}^2$ and
\[
\theta(\omega) = \left (
\begin{array}{c}
\cos(\omega)
\\
\sin(\omega)
\end{array}
\right ) \in \mathbb{S}^1.
\]
We will also write
\[
\theta_\perp = \partial_\omega \theta = \left (
\begin{array}{c}
-\sin(\omega)
\\
\cos(\omega)
\end{array}
\right ) 
\]
which corresponds to a rotation of $\theta$ by $\pi/2$ radians anti-clockwise. To save space, we will use the following notation for directional limits (where $g$ is any function)
\[
\lim_{s \rightarrow (s^*)^\pm} g(s) = g((s^*)^\pm).
\]
We next include the following classification for boundary points of $a$ and $f$.
\begin{defn}{\bf($\mathcal{B},\mathcal{P},\mathcal{P}^{e})$} \label{def,P}
	For multi-bang $a$ and p.a.b. $C^2$ $f$ with boundaries $\{\partial\Omega_{j}^a\}_{j=1}^{n}$ and $\{\partial\Omega_k^f\}_{k=1}^{m}$ respectively, which may overlap, we will write
	\[
	\mathcal{B}= \left (\bigcup_{j=1}^n \partial \Omega^a_j\right ) \bigcup\left ( \bigcup_{k=1}^m \partial \Omega^f_k\right )
	\]
	for the set of all boundary points. The set $\mathcal{P}^e$ will be the subset of $\mathcal{B}$ of points that lie in the interior of a flat edge of some boundary and finally $\mathcal{P} = \mathcal{B} \setminus \mathcal{P}^e$.
\end{defn}


In a similar manner to \cite{Holrich}, we use the following definition which classifies lines tangential to flat sections of the boundaries.

\begin{defn}{\bf($ \mathcal{K}^e$)} \label{def,K}
	Suppose that $a$ is multi-bang and $f$ is p.a.b. $C^2$ with boundaries $\mathcal{B}$ as in Definition \ref{def,P}. Let $\mathcal{K}^e$ be the subset of all directed lines in $\mathbb{R}^2$ which contain a flat edge of one of the boundaries.
\end{defn}

Throughout Section \ref{sec,thmdef}, we will use the solution of \eqref{phtranseq} which is given by
\begin{equation}\label{eq,Stefanov}
u(\x,\theta(\omega))=\int_{-\infty}^{0}f(\x+t\theta(\omega))e^{-\int_{t}^{0}a(\x+(t-\rho)\theta(\omega))\mathrm{d}\rho}\mathrm{d}t.
\end{equation}
Note that when $\x\cdot\theta(\omega)$ is sufficiently large we find that $u=R_{a}f$ \cite{Ident}. We will also use the notation
\begin{equation}\label{eq,fpmdef}
f_\pm(\x,\theta) = \lim_{s\rightarrow 0^+} f(\x\pm s \theta_\perp)
\end{equation}
and similarly for $a$
\begin{equation}\label{eq,apmdef}
a_\pm(\x,\theta) = \lim_{s\rightarrow 0^+} a(\x\pm s \theta_\perp).
\end{equation}
The jumps of $f$ and $a$ across a curve tangent to $\theta$ at $\x$ are then given by
\begin{equation}\label{eq,fajumps}
\begin{split}
\Delta_\pm f(\x,\theta) & =\pm\Big (f_-(\x,\theta)-f_+(\x,\theta) \Big ),\\
\Delta_\pm a(\x,\theta) & = \pm\Big (a_-(\x,\theta)-a_+(\x,\theta)\Big ).
\end{split}
\end{equation}
In order to make the formulas in Section \ref{sec,thmdef} easier to read, we will often fix a point $\x\in \mathbb{R}^2$ and use the notation
\begin{equation}\label{eq,Datechdef}
\begin{split}
R_af(\omega) & = R_af(\x,\theta(\omega)),\\
Da(t,\omega) & = Da(\x+t\theta(\omega),\theta(\omega)),\\ f(t,\omega)& = f(\x+t\theta(\omega)).
\end{split}
\end{equation}
With this notation established we are ready to present the main result of this paper.

\section{Main Theorem}\label{sec,thmdef}

We begin by stating the novel theorem given in this paper. The proof of this theorem is the focus of this section and is split into several intermediate results. 
\begin{theorem}\label{thm1}
Let $a$ be multi-bang with boundaries $\{ \partial \Omega^{a}_{j}\}_{j=1}^{n}$ and let $f$ be p.a.b. $C^2$ with boundaries $\{ \partial \Omega^{f}_{k}\}_{k=1}^{m}$ such that for any line simultaneously tangent to $\partial \Omega_j^a$ and $\partial \Omega_k^f$ at a point $\x$ which is not a corner, $\partial \Omega_j^a = \partial \Omega_k^f$ in a neighbourhood of $\x$ (i.e. there are no cusps between the boundaries). Suppose that for all $\x^* \in \mathcal{B}$ which are not corner points and with tangent line $L(\x^*,\theta^*=\theta(\omega^*))$ to $\mathcal{B}$:
\begin{equation}\label{eq,edgecancel2}
    \begin{split}
    & \Delta_+ f(\x^*,\theta^*) - \Delta_+ a(\x^*,\theta^*) \frac{u(\x^*,\theta((\omega^*)^-)) + u(\x^*,\theta((\omega^*)^+))}{2} \neq 0.
\end{split}
\end{equation}
Then we can uniquely recover $a$ and $f$ from $R_{a}f$. 
\end{theorem}
\begin{proof}
The proof for this is given in section \ref{subsec,boundarydetermine}.
\end{proof}

\begin{remark} \label{rem:Ham}
We point out that if $L(\x^*,\theta^*) \notin \mathcal{K}^e$ then \eqref{eq,edgecancel2} becomes
\begin{equation}\label{eq,nonzerocurvcancel}
    \Delta_{+}f(\x^*,\theta^*)-\Delta_+ a(\x^*,\theta^*) u(\x^*,\theta^*)\neq 0.
\end{equation}
In this case, $\theta^*$ can be replaced by $-\theta^*$ which also gives a line tangent to $\mathcal{B}$ at $\x^*$. It is only possible that \eqref{eq,nonzerocurvcancel} fails with both $\theta^*$ and $-\theta^*$ if
\[
u(\x^*,\theta^*) - u(\x^*,-\theta^*) = 0
\]
which provides a restrictive condition on the shape of the boundary. Interestingly, assuming $u$ is smooth, this condition can only apply at every point $\x$ in an analytic section of $\mathcal{B}$ if that section is given by the projected Hamiltonian flow for
\[
H(\x,\xi) = u\left (\x,\frac{\xi_\perp}{|\xi|}\right ) - u\left (\x,-\frac{\xi_\perp}{|\xi|} \right ).
\]
In \cite{Ident}, these curves are also used to analyse the linearisation of the identification problem in which case they appear as bicharacteristics for a certain pseudodifferential operator. The reference \cite{Ident} also gives several examples of these bicharacteristic curves in particular cases.

In fact, the requirements given in \eqref{eq,edgecancel2} and \eqref{eq,nonzerocurvcancel} are closely related to the result obtained in \cite[Prop 6.2]{Ident} which shows that, if $\Delta_+ f$ and $\Delta_+ a$ are replaced by the perturbations in a linearised version of the problem, then the quantity in \eqref{eq,nonzerocurvcancel} can be stably recovered.
\end{remark}

\begin{remark}
The hypothesis excluding cusps between the boundaries is a technical requirement which we suspect can be removed but in doing so we would likely need to modify \eqref{eq,edgecancel2} at the cusp points. It also seems likely that the technical requirement of piecewise analytic boundaries can be relaxed to piecewise smooth boundaries although it is not done in this work. Note that the analyticity of the boundaries is used in the proof of Lemma \ref{lem,multtang} when considering points at which the boundary has zero curvature and is also implicitly used in the proof of Theorem \ref{thm1} to eliminate pathological cases such as lines with infinitely many tangent points.
\end{remark}

\begin{remark}\label{rem,densereq}

The method behind the proof of Theorem \ref{thm1} is an extension of the techniques used in \cite{Holrich}. In \cite{Holrich}, the unique recovery of $a$ and $f$ is underpinned by the ability to determine those parts of the boundaries between regions of constant attenuation which are line segments or have non-zero curvature. The collection of all such points determines the boundaries. However, unlike the cases examined in \cite{Holrich}, it is possible for the quantity in \eqref{eq,nonzerocurvcancel} to be zero in which case the boundary in question cannot be determined by leading order singularities in the data whether or not the curvature is zero. Note that in general points on the boundary for which this cancellation occurs can be arbitrarily close to one another or even open sets of the boundary (although see remark \ref{rem:Ham} above). For example consider the radial case given in Figure \ref{fig,circlegroundtruth} of Section \ref{sec,NumericalExamples}. With a suitable choice of multi-bang $a$ and $f$, this choice is discussed in detail in section \ref{sec,NumericalExamples}, \eqref{eq,nonzerocurvcancel} fails on the entire boundary of $a$. Alternately, if we again take the radial example given in Figure \ref{fig,circlegroundtruth}, then it is possible to choose non-constant $f$ on the inner disk so that \eqref{eq,nonzerocurvcancel} is only zero on a portion of the circle.
\end{remark}

\subsection {Behaviour across boundaries}\label{subsec,jumps}

This section analyses the dependence of $R_{a}f$ on $\omega$ for rays which are tangent to boundaries for $a$ and/or $f$. The results are broken down into three cases which are considered separately in Lemmas \ref{lem,tangentpoint}, \ref{lem,multtang} and \ref{lem,edgecase}. First, in Lemma \ref{lem,tangentpoint}, we look lines tangent at single point on an analytic part of the boundaries with non-zero curvature. Lemma \ref{lem,multtang} then considers other cases with multiple points of tangency or zero curvature. Finally, in Lemma \ref{lem,edgecase}, we look at lines in $\mathcal{K}^e$.  

We begin with the case of lines tangent to a boundary at exactly one point where the curvature is not zero. For the statement of the Lemma, it is useful to introduce the notation
\begin{equation}\label{eq,Jdef}
\begin{split}
& J_\pm(\x, \theta^*=\theta(\omega^*)) = \lim_{\omega \rightarrow (\omega^*)^\pm} |\omega-\omega^*|^{1/2}  \partial_\omega R_a f(\x,\theta(\omega))
\end{split}
\end{equation}

\begin{lemma}\label{lem,tangentpoint}
Assume the same hypotheses as Theorem \ref{thm1} and that $L(\x,\theta^*)$ is tangent to a boundary at only one point $\x^*=\x+\ell \theta^*$ which is not a corner and at which the curvature of the boundary $\kappa$ is positive. If $s=\pm 1$ is such that $\ell s\theta^*_\perp$ points to the convex side of the boundary at $\x^*$, then
    \begin{equation}\label{eq,dwRjump}
   J_s(\x,\theta^*) = \sqrt{\frac{2 |\ell|}{\kappa}} \Big (\Delta_- f(\x^*,\theta^*) - \Delta_- a(\x^*,\theta^*) u(\x^*,\theta^*) \Big ) e^{-Da(\x^*,\theta^*)}
    \end{equation}
and $J_{-s}(\x,\theta^*) = 0$. When $\ell=0$, \eqref{eq,dwRjump} is true (and equals zero) for $s=\pm 1$. Furthermore, if $J_s(\x,\theta^*)\neq 0$, then
\begin{equation}\label{eq,ajump}
\begin{split}
&
\Delta_{+}a(\x^*,\theta^*)=\\
&\hskip.75cm \frac{\lim_{\omega \rightarrow (\omega^*)^{s}} \Bigg [ |\omega-\omega^*| \partial_\omega^2 R_a f(\omega) +\frac{1}{2} \Big ( \partial_\omega R_a f(\omega) - \partial_\omega R_a f(2\omega^*-\omega) \Big ) \Bigg ] }{\sqrt{\frac{2 |\ell|}{\kappa}} J_s(\x,\theta^*)}.
\end{split}
\end{equation}
Note that in \eqref{eq,ajump} we are using the notation given in \eqref{eq,Datechdef} with $\x$ fixed.

\end{lemma}

\begin{proof}
After translation, rotation and possible reflection we assume without loss of generality that $\omega^* = 0$, $s=1$, $\ell\geq 0$ and $\x^*$ is the origin which implies $\x = (-\ell,0)$. In this case, it is straightforward to see that $J_{-}(\x,\theta^*) = 0$ since for small negative $\omega$ all points of intersection between $L(\x,\theta(\omega))$ and the boundaries will have bounded derivatives (you can see this by following the proof below but omitting $t_\pm$). Thus we consider only $J_+$ and choose $\epsilon$ sufficiently small so that for $0<\omega<\epsilon$ $L(\x,\theta(\omega))$ intersects the boundaries of $a$ and $f$ at a fixed number of points given by $\{t_i(\omega)\}_{i=1}^N$ and $\{\t_i(\omega)\}_{i=1}^{M}$ respectively. For the remainder of this proof we will assume $\omega$ is in this range of values and the shared point of tangency corresponds to $t_+(\omega)$ and $t_-(\omega)$ which occur in both sets of intersection points.

We write $f$, using the notation from \eqref{eq,Datechdef} with $\x$ fixed, as
\[
f(t,\omega) = \sum_{j=1}^M (f_j(t,\omega)-f_{j-1}(t,\omega)) \chi_{\{t>\t_j\}}(t) =  \sum_{j=1}^M \Delta_j f(t,\omega) \chi_{\{t>\t_j\}}(t)
\]
where each $f_j \in C_c^2(\mathbb{R}^2)$ with $f_0 = f_M = 0$ and we have defined $\Delta_j f = f_j-f_{j-1}$. We will also write $f_\pm$ for the functions $f$ on either side of the tangent point ($f_+$ is the value between $t_-$ and $t_+$) and so $\Delta_\pm f(\x^*,\theta^*) = \pm (f_-(\x^*) - f_+(\x^*))$ (recall \eqref{eq,fajumps}). Similarly, we write $a$ as
\[
a(t,\omega) = \sum_{j=1}^N (a_{j-1}-a_{j}) \chi_{\{t<t_j\}}(t) = -\sum_{j=1}^N \Delta a_j \chi_{\{t<t_j\}}(t)
\]
where the $\{a_j\}_{j=1}^N$ are constants and $\Delta_j a = a_j-a_{j-1}$. Furthermore, we define $a_\pm$ as we defined $f_\pm$ and so $\Delta_\pm a(\x^*,\theta^*) =\pm(a_--a_+)$.

For the next steps we have to consider $\ell = 0$ as a special case. First, when $\ell > 0$, using the analyticity of the boundary near $\x^*$, we know that for $k_0 = \sqrt{\frac{\ell}{2 \kappa}}$ and a constant $k_1$ the following asymptotic formulae hold as $\omega \rightarrow 0^+$
\begin{equation}
\begin{split}
t_+(\omega) &= \ell + 2k_0 \omega^{1/2} + k_1 \omega+ \mathcal{O}(\omega^{3/2}),\\
t_-(\omega) & = \ell - 2k_0 \omega^{1/2} + k_1 \omega+ \mathcal{O}(\omega^{3/2}),\label{t}
\end{split}
\end{equation}
\begin{equation}
\begin{split}
\partial_\omega t_+(\omega) & = k_0 \omega^{-1/2} + k_1+ \mathcal{O}(\omega^{1/2}), \\ \partial_\omega t_-(\omega) & = -k_0 \omega^{-1/2} +k_1+ \mathcal{O}(\omega^{1/2}), \label{dt}
\end{split}
\end{equation}
\begin{equation}
\begin{split}
\partial_\omega^2 t_+(\omega) & = -\frac{k_0}{2} \omega^{-3/2} + \mathcal{O}(\omega^{-1/2}), \\ \partial_\omega^2 t_-(\omega) & = \frac{k_0}{2} \omega^{-3/2} + \mathcal{O}(\omega^{-1/2}).\label{d2t}
\end{split}
\end{equation}
When $\ell = 0$ we have instead that $t_-(0) = 0$ and $t_+$ is smooth up to $\omega = 0$ and vanishes at $\omega = 0$.

Using the notation given in \eqref{eq,Datechdef}, the beam transform of $a$ can written
\[
Da(t,\omega) = \sum_{i=1}^N (a_{i-1}-a_i) \phi_i(t,\omega)
\]
where
\[
\phi_i(t,\omega)
=
\left \{ \begin{array}{cl}
t_i(\omega), & t< t_i(\omega),\\
t, & t_i(\omega)\leq t.
\end{array}
\right .
\]
Note that
\begin{equation}
\begin{split}
&-Da(t,\omega) = -Da(t,0) \\
& \hskip.25in + (a_- - a_+)\Big ((t_+-t_-) \chi_{\{t<t_-\}}(t) + (t_+-t) \chi_{\{t_-<t<t_+\}} \Big )
+ \mathcal{O}(\omega)
\end{split}
\end{equation}
which gives
\begin{equation}\label{eomegaas}
e^{-Da(t,\omega)} = e^{-Da(t,0)} + \mathcal{O}(\omega^{1/2}).
\end{equation}
Using the same type of reasoning we can also show that
\[
e^{-Da(t,-\omega)} = e^{-Da(t,0)} + \mathcal{O}(\omega) 
\]
and combining these together gives
\begin{equation} \label{e-omegadiff}
e^{-Da(t,\omega)}-e^{-Da(t,-\omega)} = \mathcal{O}(\omega^{1/2}).
\end{equation}
The derivative of the beam transform with respect to $\omega$ is 
\begin{equation}\label{dDa}
\partial_\omega Da(t,\omega) = -\sum_{i=1}^N \Delta_i a \ \partial_\omega t_i\ \chi_{\{t<t_i\}}(t)
\end{equation}
and $\partial_\omega D a(t,-\omega)$ is the same except without the two terms involving $t_\pm$. 

We next consider the attenuated ray transform as a function of $\omega$ which is given by 
\begin{equation}\label{Raf}
R_af(\omega) = \sum_{j=1}^M \int_{\t_j}^\infty \Delta_j f(t,\omega) e^{-Da(t,\omega)} \ \mathrm{d} t.
\end{equation}
The derivative is
\begin{equation} \label{dRaf}
\begin{split}
\partial_\omega R_a f(\omega) & = -\sum_{j=1}^M \partial_\omega \t_j \ \Delta_jf(\t_j,\omega) e^{- Da(\t_j,\omega)}\\
& \hskip0cm + \sum_{j=1}^M \int_{\t_j}^\infty \Big ( \partial_\omega \Delta_j f(t,\omega) - \partial_\omega Da(t,\omega) \Delta_j f(t,\omega) \Big )e^{-Da(t,\omega)} \ \mathrm{d} t.
\end{split}
\end{equation}
Using \eqref{dDa} we have
\begin{equation}\label{dRaaff}
\begin{split}
\partial_\omega R_a f(\omega) & =  -\sum_{j=1}^M \partial_\omega \t_j \ \Delta_jf(\t_j,\omega) e^{- Da(\t_j ,\omega)} \\
& \hskip0.5in + \sum_{j=1}^M \int_{\t_j}^\infty  \partial_\omega \Delta_j f(t ,\omega) e^{-Da(t,\omega)} \ \mathrm{d} t\\
& \hskip0.5in +  \sum_{\{ j,i \ | \ \t_j< t_i \}} \Delta_i a\ \partial_\omega t_i \int_{\t_j}^{t_i} \Delta_j f(t,\omega)e^{-Da(t,\omega)} \ \mathrm{d} t.
\end{split}
\end{equation}
Also, $\partial_\omega R_af (-\omega)$ is given by the same formula, but with the terms corresponding to $t_\pm$ removed from all sums except when $\ell = 0$. In the case $\ell = 0$, all terms in \eqref{dRaaff} are $\mathcal{O}(1)$ and the same is true when we use $-\omega$. This completes the proof when $\ell =0$ since it shows $\partial_\omega R_a f(\pm \omega) = \mathcal{O}(1)$ and so the limit defining $J_+(\x,\theta^*)$ will be zero. From now on we only consider the case $\ell >0$. Note that, in this case, the terms in the second sum in \eqref{dRaaff} corresponding to $t_\pm$ are, when taken together,
\[
\int_{t_-}^{t_+} \partial_\omega \Delta_- f(t,\omega) e^{-Da(t,\omega)} \ \mathrm{d} t=\mathcal{O}(\omega^{1/2})
\]
because of \eqref{t}, and so
\begin{equation}\label{dRaaffmin}
\begin{split}
&\partial_\omega R_a f(\omega) = \partial_\omega R_a f(-\omega) \\
&\hskip.5cm - \partial_\omega t_- \ \Delta_-f(t_-,\omega) e^{- Da(t_- ,\omega)} -  \partial_\omega t_+ \ \Delta_+f(t_+,\omega) e^{- Da(t_+ ,\omega)}\\
&\hskip1cm + \sum_{\{j \ | \ \t_j < t_-\}} \Delta_- a\ \partial_\omega t_- \int_{\t_j}^{t_-} \Delta_j f(t,\omega)e^{-Da(t,\omega)} \ \mathrm{d} t\\
&\hskip1.5cm + \sum_{\{j \ | \ \t_j < t_+\}} \Delta_+ a\ \partial_\omega t_+ \int_{\t_j}^{t_+} \Delta_j f(t,\omega)e^{-Da(t,\omega)} \ \mathrm{d} t + \mathcal{O}(\omega^{1/2}).
\end{split}
\end{equation}
The sums may be eliminated in the second and third lines to write the previous equation in the simpler form
\begin{equation}
\begin{split}
& \partial_\omega R_a f(\omega) = \partial_\omega R_a f(-\omega) \\
&\hskip1cm - \partial_\omega t_- \ \Delta_-f(t_-,\omega) e^{- Da(t_- ,\omega)} -  \partial_\omega t_+ \ \Delta_+f(t_+,\omega) e^{- Da(t_+ ,\omega)}\\
&\hskip1.5cm + \Delta_-a\ \partial_\omega t_- \int_{-\infty}^{t_-}  f(t,\omega)e^{-Da(t,\omega)} \ \mathrm{d} t \\
&\hskip2cm + \Delta_+ a\ \partial_\omega t_+ \int_{-\infty}^{t_+} f(t,\omega)e^{-Da(t,\omega)} \ \mathrm{d} t + \mathcal{O}(\omega^{1/2}).
\end{split}
\end{equation}
Furthermore, using \eqref{t} and \eqref{dt} we obtain
\begin{equation}\label{eq,usefuleq}
\begin{split}
&\partial_\omega R_a f(\omega) - \partial_\omega R_a f(-\omega) =\\
& k_0 \omega^{-1/2} \Big ( \Delta_-f(t_-,\omega) e^{- Da(t_- ,\omega)} + \Delta_- f(t_+,\omega)e^{- Da(t_+ ,\omega)} \Big )\\
& - k_0 \omega^{-1/2} \Delta_- a\ \Bigg(  \int_{-\infty}^{t_-} f(t,\omega)e^{-Da(t,\omega)} \ \mathrm{d} t + \int_{-\infty}^{t_+} f(t,\omega)e^{-Da(t,\omega)} \ \mathrm{d} t \Bigg ) \\
& \hskip.5cm + \mathcal{O}(\omega^{1/2}).
\end{split}
\end{equation} 
Based on \eqref{eq,usefuleq} we have the limit
\begin{equation}\label{dRaafflim}
\begin{split}
& \lim_{\omega\rightarrow 0^+} \omega^{1/2} \Big [ \partial_\omega R_a f(\omega) - \partial_\omega R_af(-\omega) \Big ] \\
&\hskip1.5cm = 2k_0 \Delta_- f(\ell,0) e^{-Da(\ell,0)} - 2 k_0 \Delta_- a\ \int_{-\infty}^\ell f(t,0) e^{-Da(t,0)} \ \mathrm{d} t.
\end{split}
\end{equation}
This proves formula \eqref{eq,dwRjump} since $\lim_{\omega\rightarrow 0^+} \omega^{1/2} \partial_\omega R_af(-\omega)=0$. 

Now we take the derivative of \eqref{dRaaffmin} and consider only the terms more singular than $\mathcal{O}(\omega^{-1/2})$. The derivative of $\partial_\omega R_a f(-\omega)$ contributes no such terms and so we just consider the derivatives of the other terms. To help manage the calculation, we will consider the derivatives of the other terms separately. First let us consider the first line in \eqref{dRaaffmin}
\begin{equation}\label{A}
A(\omega) = -\partial_\omega t_- \ \Delta_-f(t_-,\omega) e^{- Da(t_- ,\omega)} -  \partial_\omega t_+ \ \Delta_+f(t_+,\omega) e^{- Da(t_+ ,\omega)}.
\end{equation}
For the derivative of the beam transforms note that
\begin{align*}
-Da(t_-,\omega) & = t_- a_+ + \Delta_+ a\ t_+ + \sum_{\{i \ | \ t_i > t_+\}} \Delta_i a_i\ t_i,\\
-Da(t_+,\omega) & = t_+ a_-  + \sum_{\{i \ | \ t_i > t_+\}} \Delta_i a\ t_i
\end{align*}
and so
\begin{align*}
-\partial_\omega Da(t_-,\omega) & = a_+ \partial_\omega t_- + \Delta_+ a\ \partial_\omega t_+ + \mathcal{O}(1),\\
-\partial_\omega Da(t_+,\omega) & = a_- \partial_\omega t_+  + \mathcal{O}(1).
\end{align*}
Using these formulae to take the derivative of $A$ from \eqref{A}, we get
\begin{equation}
\begin{split}
\partial_\omega A(\omega) & = - \partial_\omega^2 t_- \Delta_- f(t_-,\omega) e^{-Da(t_-,\omega)} + \partial_\omega^2 t_+ \Delta_- f(t_+,\omega) e^{-Da(t_+,\omega)} \\
& \hskip1cm - \Big ( a_+\partial_\omega t_- + \Delta_+ a\ \partial_\omega t_+ \Big )\partial_\omega t_-  \Delta_- f(t_-,\omega)  e^{-Da(t_-,\omega)} \\
&\hskip1.5cm + a_-(\partial_\omega t_+)^2 \Delta_- f(t_+,\omega) e^{-Da(t_+,\omega)}+ \mathcal{O}(\omega^{-1/2})
\end{split}
\end{equation}
and using \eqref{dt} and \eqref{d2t} this becomes
\begin{equation}\label{dA}
\begin{split}
\partial_\omega A(\omega) & = -\frac{k_0}{2} \omega^{-3/2} \Bigg ( \Delta_- f(t_-,\omega) e^{-Da(t_-,\omega)} +\Delta_- f(t_+,\omega) e^{-Da(t_+,\omega)} \Bigg )\\
&\hskip2cm - 2 k_0^2 \omega^{-1}  \Delta_- a\  \Delta_- f(\ell, \omega)  e^{-Da(\ell,\omega)} + \mathcal{O}(\omega^{-1/2}).
\end{split}
\end{equation}
Now let us consider the last two lines in \eqref{dRaaffmin}:
\[
\begin{split}
B(\omega) &= \sum_{\{j \ | \ \t_j < t_-\}} \Delta_- a\ \partial_\omega t_- \int_{\t_j}^{t_-} \Delta_j f(t,\omega)e^{-Da(t,\omega)} \ \mathrm{d} t\\
&\hskip2cm + \sum_{\{j \ | \ \t_j < t_+\}} \Delta_+ a \ \partial_\omega t_+ \int_{\t_j}^{t_+} \Delta_j f(t,\omega)e^{-Da(t,\omega)} \ \mathrm{d} t.
\end{split}
\]
Differentiating, eliminating the sums as before and using \eqref{dt} and \eqref{d2t} we get
\begin{equation}\label{dB}
\begin{split}
& \partial_\omega B(\omega)  =\\
&\hskip.5cm \frac{k_0}{2} \Delta_- a\ \omega^{-3/2}\Bigg ( \int_{-\infty}^{t_-} f(t,\omega) e^{-Da(t,\omega)}\ \mathrm{d} t  +\int_{-\infty}^{t_+} f(t,\omega) e^{-Da(t,\omega)} \ \mathrm{d} t \Bigg )\\
&\hskip1cm - 2 k_0^2 \omega^{-1} \Delta _- a\ \Delta_- f(\ell,\omega) e^{-Da(\ell,\omega)}\\
&\hskip1.5cm + 4k_0^2(\Delta_- a)^2\omega^{-1} \int_{-\infty}^{\ell}f(t,\omega) e^{-Da(t,\omega)} \ \mathrm{d} t+ \mathcal{O}(\omega^{-1/2}).
\end{split}
\end{equation}
Combining \eqref{dA} and \eqref{dB}, and also using \eqref{eq,usefuleq}, gives
\[
\begin{split}
\omega \partial_\omega^2 R_af(\omega) & =  - \frac{1}{2} \Big ( \partial_\omega R_a f(\omega) - \partial_\omega R_a f(-\omega) \Big )  - 4 k_0^2  \Delta_- a\  \Delta_- f(\ell, \omega)  e^{-Da(\ell,\omega)}\\
&+ 4k_0^2(\Delta_- a)^2 \int_{-\infty}^{\ell}f(t,\omega) e^{-Da(t,\omega)} \ \mathrm{d} t+ \mathcal{O}(\omega^{1/2})
\end{split}
\]
from which we deduce the limit
\begin{equation}\label{d2Raflim3}
\begin{split}
&\lim_{\omega \rightarrow 0^+} \Bigg [ \omega \partial_\omega^2 R_a f(\omega) +\frac{1}{2} \Big ( \partial_\omega R_a f(\omega) - \partial_\omega R_a f(-\omega) \Big ) \Bigg ] \\
&\hskip0.5cm = -2 k_0 \Delta_- a\  \Bigg ( 2 k_0 \Delta_- f(\ell, 0)  e^{-Da(\ell,0)}- 2k_0\Delta_- a\ \int_{-\infty}^{\ell}f(t,0) e^{-Da(t,0)} \ \mathrm{d} t\Bigg ).
\end{split}
\end{equation}
By \eqref{eq,dwRjump}, the quantity in brackets on the right side of \eqref{d2Raflim3} is $J_+(\x,\theta^*)$ and so provided this is not zero we can obtain \eqref{eq,ajump}. This completes the proof of the lemma.
\end{proof}

The next lemma expands some of the results from Lemma \ref{lem,tangentpoint} to the cases of multiple tangent points along the same line and points at which the boundaries have zero curvature.

\begin{lemma}\label{lem,multtang}
Assume the same hypotheses as Theorem \ref{thm1} and suppose that $\x \notin \mathcal{P}$ and $L(\x,\theta^*) \notin \mathcal{K}^e$. Then given any $\epsilon >0$ sufficiently small either $J_+(\x+l \theta^*,\theta^*) = 0$ for all $l 
\in (-\epsilon,\epsilon)$ or $J_+(\x+l \theta^*,\theta^*) \neq 0$ for a dense set of $l \in (-\epsilon,\epsilon)$.

\end{lemma}

\begin{proof}
The proof for one case, when there is a single point of tangency along $L(\x,\theta^*)$, is already established by Lemma \ref{lem,tangentpoint}. Also, if there are no points of tangency along $L(\x,\theta^*)$ then $J_+(\x,\theta^*) = 0$ and the proof is complete. For other cases the proof follows many of the same steps as in Lemma \ref{lem,tangentpoint}. Indeed, under the hypotheses there can be a finite number of points of tangency along $L(\x,\theta^*)$ given by $\{\x^*_j = \x +\ell_j \theta^*\}_{j=1}^n$ and many of the initial steps of the proof of Lemma \ref{lem,tangentpoint} are still valid except that there can be more other than two $\{t_i\}$ and $\{\tilde{t}_j\}$ corresponding to tangent points and having unbounded derivatives as $\omega \rightarrow 0^+$. Also, there can be points at which the curvature of the boundary is zero in which case the asymptotic formulae \eqref{t} and \eqref{dt} must be changed. In general, each $t_i$ (or $\tilde{t}_j$) corresponding to a tangent point will satisfy
\begin{equation}\label{t0}
    t_i = \ell_i + k_i \ell_i^{1/m_i} \omega^{1/m_i} + \mathcal{O}(\omega^{2/m_i}),
\end{equation}
\begin{equation}\label{dt0}
    \partial_\omega t_i = k_i \ell_i^{1/m_i} \omega^{\frac{1-m_i}{m_i}} + \mathcal{O}(\omega^{\frac{2-m_i}{m_i}})
\end{equation}
where each $k_i$ is a non-zero constant depending on the boundary and for all $i$ $m_i \geq 2$ ($m_i=2$ is the non-zero curvature case). Note that we have used the analyticity of the boundary for this step.

Now, following the proof of Lemma \ref{lem,tangentpoint}, \eqref{dRaaff} is still valid. Suppose that we index the tangent points along $L(\x,\theta^*)$ at which $m_i = m$ (see \eqref{t0} and \eqref{dt0}) takes the largest value as $\{t_{i}\}$. Then we can use \eqref{dRaaff}, \eqref{t0} and \eqref{dt0} to get
\begin{equation}\label{eq,lem2}
\begin{split}
\partial_\omega R_a f(\omega) & = \sum_{i} k_{i} \ell_{i}^{1/m} \omega^{\frac{1-m}{m}}\Bigg ( \Delta_{{i}} f(t_{i}) e^{-Da(t_{i})} +  \Delta_{{i}}a\  \int_{-\infty}^{t_{i}} f(t,\omega) e^{-Da(t,\omega)}\mathrm{d} t 
\Bigg ) \\
& \hskip2cm + \mathcal{O}(\omega^{\frac{1-m}{m}-\alpha})
\end{split}
\end{equation}
for some small $\alpha >0$.  Here the $\Delta_i f(t_i)$ and $\Delta_i a$ are the respective jumps along $L(\x,\theta^*)$ across the boundary corresponding to $t_i$.

Now because of the non-cancellation condition \eqref{eq,nonzerocurvcancel}, the terms in parentheses in \eqref{eq,lem2} do not vanish as $\omega \rightarrow 0^+$ and then $J_+(\x+l\theta^*,\theta^*) \neq 0$ for all $l$ near zero except possibly at isolated values where the terms in the sum cancel.

\end{proof}


\noindent Next, we consider the lines in $\mathcal{K}^e$ which are tangent to a flat edge of a boundary in the following lemma. It is useful to introduce the following notation where, as usual, $\theta^* = \theta(\omega^*)$
\begin{equation}\label{eq,K}
\begin{split}
& K_n(\x+l \theta^*,\theta^*) = e^{-D\Delta_+a(\x,\theta(\omega^*))} \frac{\partial^n}{\partial l^n} R_a f(\x+l\theta^*,\theta(\omega^*_+)) \\
&\hskip6cm + (-1)^n \frac{\partial^n}{\partial l^n} R_a f(\x+l \theta^*,\theta(\omega^*_-)).
\end{split}
\end{equation}
We only need the cases $n =1$ and $2$.

\begin{lemma}\label{lem,edgecase}
	Let $f$ be p.a.b. $C^2$ and $a$ be multi-bang. Then, for points where $f_\pm(\x+l\theta^*,\theta^*)$ and $a_\pm(\x+l\theta^*,\theta^*)$ are continuous with respect to $l$ at $l=0$,
	\begin{equation}\label{eq,edge1}
	\begin{split}
	   &K_1(\x,\theta^*)= \\
	   & \hskip1cm 2 e^{-Da_-(\x,\theta^*)} \left ( \Delta_+ f(\x,\theta^*)-\frac{\Delta_+ a(\x,\theta^*)}{2} \left (u(\x,\theta(\omega^*_+) + u(\x,\theta(\omega^*_-)) \right ) \right )
	   \end{split}
	\end{equation}
	and
\begin{equation}\label{eq,edgeformulafin}
\begin{split}
K_1(\x,\theta^*) \Delta_- a(\x,\theta^*) = K_2(\x,\theta^*).
\end{split}
\end{equation}
	
\end{lemma}
\begin{proof}
By rotating and translating, we assume without loss of generality that $\x$ is the origin and $\omega^* = 0$. 
For $\omega$ close to zero, we then define
\[
R_af(l,\omega) = R_af(l \theta^*,\theta(\omega)) = \int_{-\infty}^\infty f(l \theta^* + t \theta(\omega)) e^{-Da(l \theta^* + t \theta(\omega),\theta(\omega))} \ \mathrm{d} t,
\]
\begin{equation}\label{eq,fapmdef}
f_\pm(l) = \lim_{\epsilon \rightarrow 0^\pm} f(l,\epsilon), \quad a_{\pm}(l) = \lim_{\epsilon \rightarrow 0^\pm} a(l,\epsilon)
\end{equation}
and
\[
Da_\pm(l,l') = \int_l^{l'} a_\pm(r) \ \mathrm{d} r.
\]
Note the definitions in \eqref{eq,fapmdef} agree with those in \eqref{eq,fpmdef} and \eqref{eq,apmdef} when we take $\x = (l,0)$ and $\theta= (1,0)$. Also, we write $\Delta_\pm a(l) = \pm (a_-(l) - a_+(l))$ and $\Delta_\pm f(l) = \pm (f_-(l)-f_+(l))$.

With these definitions, we have the following limits
\[
\begin{split}
R_a f(l,0^\pm) & = \int_{-\infty}^l f_\mp(t) e^{-Da_{\mp}(t,l)-D a_\pm(l,\infty)} \ \mathrm{d} t+ \int_{l}^\infty f_\pm(t) e^{-Da_\pm(t,\infty)} \ \mathrm{d} t.
\end{split}
\]
If we take the derivative of this limit with respect to $l$ we obtain
\begin{equation}\label{eq,edged1}
\begin{split}
\frac{\partial}{\partial l} R_a f(l,0^\pm) & =\mp\Delta_+ a(l) \int_{-\infty}^l f_\mp( t) e^{-Da_{\mp}(t,l)-D a_\pm(l,\infty)} \ \mathrm{d} t \\
&\hskip3cm \pm \Delta_+ f(l) e^{- Da_\pm(l,\infty)}.
\end{split}
\end{equation}
The second derivative with respect to $l$ is
\begin{equation}\label{eq,edged2}
\begin{split}
\frac{\partial^2}{\partial l^2} R_a f(l,0^\pm) & =(\Delta_+ a(l))^2 \int_{-\infty}^l f_\mp( t) e^{-Da_{\mp}(t,l)-D a_\pm(l,\infty)} \ \mathrm{d} t\\
& \hskip1cm \mp \Delta_+ a(l) \ f_\mp(l) e^{- Da_\pm(l,\infty)}  \\
&\hskip2cm \pm \frac{\partial (\Delta_+ f)}{\partial l}(l) e^{- Da_\pm(l,\infty)}\\
& \hskip3cm \pm a_\pm(l) \Delta_+ f(l) e^{- Da_\pm(l,\infty)}.
\end{split}
\end{equation}
Formula \eqref{eq,edged1} gives
\begin{equation}\label{eq,edgecancelcond}
\begin{split}
&e^{-D\Delta_+a(l,\infty)} \frac{\partial}{\partial l} R_a f(l,0^+) - \frac{\partial}{\partial l} R_a f(l,0^-)  \\
&\hskip.25cm =2 e^{-Da_- (l,\infty)} \left (\Delta_+ f(l) - \frac{\Delta_+ a(l)}{2} \int_{-\infty}^l f_+(t) e^{-Da_+(t,l)}+f_-(t) e^{-Da_-(t,l)} \mathrm{d} t \right ),
\end{split}
\end{equation}
which, recalling the formula \eqref{eq,Stefanov} for $u$, implies \eqref{eq,edge1}. Similarly, formula \eqref{eq,edged2} implies \eqref{eq,edgeformulafin} which completes the proof.
\end{proof}

Assuming that a suitable number of points in the boundaries of $a$, do not violate \eqref{eq,edgecancel2}, Lemmas \ref{lem,tangentpoint}, \ref{lem,multtang} and \ref{lem,edgecase} give us the tools we need to be able to uniquely determine the jumps of $a$ and thus prove Theorem \ref{thm1}. The next subsection shows how this can be done.

\subsection {Determination of the jumps of $a$}\label{subsec,boundarydetermine}

This subsection completes the proof of Theorem \ref{thm1} by showing, in the next Lemma, that all jumps of $a$ can be determined. 

\begin{lemma}
\label{lem,findboundary}
Let $a$ be multi-bang and $f$ be p.a.b. $C^2$ be such that the hypothesis of Theorem \ref{thm1} are satisfied. Then $R_a f$ determines $\mathcal{B}$ as well as the jumps in $a$ at every boundary.
\end{lemma}

\begin{proof}
The determination proceeds in several steps. The first step is to find all of the flat edges contained in the boundaries. Here we will use slightly modified versions of the functions $K_n$ introduced in \eqref{eq,K}. Indeed, we will inductively determine $\Delta_+ a$ along each line by replacing $\Delta_+ a$ in \eqref{eq,K} using our current knowledge $\Delta_+ a^c$ at each step ($c$ is not an index here but rather a notation indicating the current information about the jump of $a$). The induction along each line starts with $\Delta_+ a^c = 0$. With this in mind, we introduce 
\begin{equation}\label{eq,Kj}
\begin{split}
& K^c_n(\x+l \theta^*,\theta^*) = e^{-D\Delta_+a^c(\x,\theta(\omega^*))} \frac{\partial^n}{\partial l^n} R_a f(\x+l\theta^*,\theta(\omega^*_+)) \\
&\hskip6cm + (-1)^n \frac{\partial^n}{\partial l^n} R_a f(\x+l \theta^*,\theta(\omega^*_-)).
\end{split}
\end{equation}
Let us now fix a particular choice of $\x$ and $\theta^*$ and determine $\Delta_+ a$ along the line $L(\x,\theta^*)$. Along this line there will be an increasing set of points, possibly empty, given by $\{\beta_i\}_{i=1}^N$ such that $f_\pm(\x+l \theta^*)$ and $a_\pm(\x+l \theta^*)$ are continuous with respect to $l$ except for removable singularities and at the points $\{\beta_i\}_{i=1}^N$. We start from large $l$ and work backwards inductively determining each $\beta_i$ as well as $\Delta_+ a$. Indeed, for $l > \beta_N$ it must be true that $\Delta_+ a(\x+l \theta^*,\theta^*) = 0$ and so
\[
K^c_1(\x+l\theta^*,\theta^*) = K_1(\x+l \theta^*,\theta^*) = 0.
\]
In the case when there are no jumps over line segments along $L(\x,\theta^*)$, this is true for all $l$. On the other hand, assuming there are jumps over line segments,
\[
\lim_{l \rightarrow \beta_N^-} K^c_1(\x+l\theta^*,\theta^*)=\lim_{l \rightarrow \beta_N^-} K_1(\x+l\theta^*,\theta^*) \neq 0.
\]
Note that we have used the assumption \eqref{eq,edgecancelcond} and equation \eqref{eq,edge1} for this last conclusion. Thus we can determine if there are no jumps over line segments or if there are such jumps then we can determine $\beta_N$ from
\[
\beta_N = \sup \{ l \ : \ K_1^c(\x+l \theta^*,\theta^*) \neq 0\}.
\]
Furthermore, equation \eqref{eq,edgeformulafin} allows us to determine $\Delta_+ a(\x+\beta_N^- \theta^*,\theta^*)$ by
\[
\Delta_+ a(\x+\beta_N^- \theta^*,\theta^*) = \frac{K_2^c(\x+\beta_N^-\theta^*,\theta^*)}{K_1^c(\x+\beta_N^-\theta^*,\theta^*)} = \frac{K_2(\x+\beta_N^-\theta^*,\theta^*)}{K_1(\x+\beta_N^-\theta^*,\theta^*)}
\]
This completes the base case of the induction.

Now suppose by induction that, for some $i>1$, we have found $\beta_i$, $\Delta_+a(\x+l \theta^*,\theta^*)$ for $l> \beta_i$ and $\Delta_+a(\x+\beta_i^-\theta^*,\theta^*)$. Then we define
\[
\Delta_+ a^{c}(\x + l \theta^*, \theta^*) = 
\left \{ 
\begin{array}{cl}
\Delta_+ a(\x + l \theta^*,\theta^*), & \mbox{if $l\geq \beta_i$,}\\
\Delta_+ a(\x + \beta_i^- \theta^*,\theta^*),  & \mbox{if $l< \beta_i$.}
\end{array}
\right .
\]
Note that this only uses information we already know from induction and $\Delta_+ a^{c}(\x + l \theta^*, \theta^*)=\Delta_+ a(\x + l \theta^*, \theta^*)$ for $l> \beta_{i-1}$. This definition implies
\[
K^{c}_n(\x+l\theta^*,\theta^*) = K_n(\x+l\theta^*,\theta^*) 
\]
$l > \beta_{i-1}$ and also in the limit $l = \beta_{i-1}^-$. Now there are two possible case that must be considered.

First, if $K_1(\x+l\theta^*,\theta^*) = 0$ for $\beta_{i-1}<l<\beta_i$, then $\beta_{i-1}$ can be determined from
\[
\beta_{i-1} = \sup \{ l<\beta_i \ : \ K_1^{c}(\x+l \theta^*,\theta^*) \neq 0\}
\]
and finally
\[
\Delta_+a(\x+\beta_{i-1}^-\theta^*,\theta^*) = \frac{K^{c}_2(\x+\beta_{i-1}^- \theta^*,\theta^*)}{K^{c}_1(\x+\beta_{i-1}^- \theta^*,\theta^*)}
\]
which completes the induction step in this case.

Second, if $K_1(\x+l\theta^*,\theta^*) \neq 0$ for $\beta_{i-1}<l<\beta_i$ then we consider
\[
\beta = \sup \left \{ l < \beta_i \ : \ \Delta_+ a(\x + \beta_i^- \theta^*,\theta^*) \neq \frac{K^{c}_2(\x+l\theta^*)}{K^{c}_1(\x+l\theta^*)} \right \}.
\]
Then $\beta = \beta_{i'}$ for some $i' < i$ (note that if $\Delta_+ a(\x+l\theta^*,\theta^*)$ does not jump at $\beta_{i-1}$ then $\beta< \beta_{i-1}$). Having found the next point, $\beta_{i'}$ at which $\Delta_+ a(\x+l\theta^*,\theta^*)$ jumps we also get
\[
\Delta_+a(\x+\beta_{i'}^-\theta^*,\theta^*) = \frac{K^{c}_2(\x+\beta_{i'}^- \theta^*,\theta^*)}{K^{c}_1(\x+\beta_{i'}^- \theta^*,\theta^*)}
\]
which completes the induction step in the second case.

Now, by induction we have determined $\Delta_+ a(\x + l \theta^*,\theta^*)$ for $l>\beta_1$ and for $l<\beta_1$ the jump is zero. Therefore we have determined the jumps of $a$ along the entire line $L(\x,\theta^*)$. Applying this to all $\x$ and $\theta^*$ we can determine all flat edges of the boundaries (i.e. the set $\mathcal{P}^e$) as well as the jump in $a$ across each of these flat boundaries.

Now we determine the other points in the boundaries of $a$. For this, consider $\x \in \mathcal{P}$. By perturbing $\x$ slightly we can always obtain a point $\tilde{\x}$ in an analytic section of a single boundary such that the line tangent to the boundary at $\tilde{\x}$ is not tangent to any boundary at any other point. Thus the set of all such $\tilde{\x}$ is dense in $\mathcal{P}$ and we will call this set $\widetilde{\mathcal{P}}$. For now suppose $\x \in \widetilde{\mathcal{P}}$ with the line in direction $\theta^*$ tangent to the boundary at $\x$ such that $\theta^*_\perp$ points towards the convex side of the boundary. Then by Lemma \ref{lem,tangentpoint} $J_+(\x+l \theta^*,\theta^*) \neq 0$ for $l<0$ and $J_+(\x+l \theta^*,\theta^*) = 0$ for $l\geq 0$. On the other hand, if $\x \notin \mathcal{P}$ is such that $L(\x,\theta^*) \notin \mathcal{K}^e$, then Lemma \ref{lem,multtang} says that for $\epsilon$ sufficiently small, either $J_+(\x+l\theta^*,\theta^*)=0$  all $l \in (-\epsilon,\epsilon)$ or $J_+(\x+l \theta^*,\theta^*)\neq 0$ for a dense set of $l \in (-\epsilon,\epsilon)$. From the above argument, if we define the set
\[
\begin{split}
&\mathcal{G} = \Big \{ \x \ : \ \exists\ \theta^*\ \mathrm{s.t.}\  L(\x, \theta^*) \notin \mathcal{K}^e \ \mbox{and}\ J_+(\x+ \ell \theta^*,\theta) = 0 \ \mbox{when $l\geq 0$},\\
& \hskip2cm J_+(\x+ \ell \theta^*,\theta) \neq 0 \ \mbox{when $l< 0)$}
\Big \}
\end{split}
\]
then $\widetilde{\mathcal{P}} \subset \mathcal{G} \subset \mathcal{P}$. Since $\mathcal{G}$ can be determined from $R_af$ and $\mathcal{K}^e$, which we already constructed above, we can determine $\mathcal{G}$ and combining the arguments above we see that the set of all boundary points is given by
\[
\mathcal{B} = \overline{\mathcal{G}} \cup \mathcal{P}^e.
\]
Therefore we can construct all of the boundaries. Once the boundaries are constructed we can use equation \eqref{eq,ajump} from Lemma \ref{lem,tangentpoint} to determine the jump in $a$ at all $\x \in \mathcal{P}$. This finally allows us to determine the jump in $a$ across all boundaries and so completes the proof.
\end{proof}

Having determined all boundaries and jumps in $a$, the proof of Theorem \ref{thm1} follows immediately.

\begin{proof}[Proof of Theorem 1]

By Lemma \ref{lem,findboundary} we can recover the boundaries of $a$ and the jumps in $a$ across each boundary. As $a$ is compactly supported, and therefore $0$ outside the outer most boundaries, we can use the jumps to assign a unique value in each region and therefore uniquely determine $a$. 

Finally making use of the Novikov inversion formula\cite{inversion,NovikovRange} for the AtRT for known $a$ also yields $f$, as required.
\end{proof}

This concludes the theoretical part of the paper. Our earlier work in \cite{Holrich} gives a joint reconstruction algorithm for multi-bang $a$ with known admissible set and the following two sections of this paper outline the algorithm as well as examine the data produced from various multi-bang $a$ and p.a.b. $C^2$ $f$. In particular we show that cancelling cases are visible in the data for both the tangent points and edges, although it appears in the examples we have studied that it is still possible to obtain reconstructions even in the cancelling cases.

We now give an outline of the numerical method given in \cite{Holrich} used to produce joint reconstructions from $R_{a}f$ and then apply this method in cases when the cancelling condition required for Theorem \ref{thm1} either does or doesn't hold.

\section{Numerical Method}\label{sec,NumericalMethod}
This section relates to the numerical method for joint recovery of $a$ and $f$ from $R_{a}f$. We follow the same method as in \cite{Holrich} where, for a domain of interest $\Omega$, we aim to solve the variational problem 

\begin{equation}\label{eq,ContVariationalprob}
\mathrm{argmin}_{a,f\in BV(\Omega)} \|R_{a}f-d\|^{2}_{2}+\alpha \mathcal{M}(a)+\lambda \mathrm{TV}(a)+\eta\mathrm{TV}(f)
\end{equation}
where $BV(\Omega)$ is the space of functions of bounded variation on $\Omega$, see \cite{BV}, $\mathrm{TV}$ is the total variation and $\mathcal{M}$ is the multi-bang regularizer, see Definition \ref{def,MBreg}. If we split the domain $\Omega$ into $M^2$ square pixels of resolution $\mathrm{d}x$ which are lexicographically ordered from top left to bottom right and assume that $a$ and $f$ are piecewise constant then we can obtain an exact formula for $R_{a}f$, which we denote $R[a]f$ and is given in \cite{Holrich}, in the discrete case.

As in \cite{Holrich} we use the following weakly convex multi-bang regularizer.
\begin{defn}[Multi-bang regularizer]\label{def,MBreg} Let the set of admissible attenuation values be ${\mathcal{A}}:=\{a_0,a_1,...,a_n\}$ with $a_{0}<a_1<...<a_n$.  The multi-bang regularizer is given by
\begin{equation}\label{eq,contglobalMB}
\mathcal{M}(a):=\int_{\Omega}m(a({\boldsymbol{x}}))\mathrm{d}{\boldsymbol{x}}
\end{equation}
where
\begin{equation}\label{eq,multibangpointwise}
m(t)=
\left \{
\begin{array}{cl}
(a_{i+1}-t)(t-a_{i}), & \mbox{{if} $t\in[a_{i},a_{i+1}]$ {for some $i$}}\\
\infty, &\mathrm{otherwise}.
\end{array}
\right .
\end{equation}
{We refer to $m$ as the pointwise multi-bang penalty. It is 0 for all multi-bang values and then behaves quadratically when $t$ is between two adjacent $a_i$.}
\end{defn}
As $a$ and $f$ are piecewise constant over pixels, and absorbing the area of each pixel $(\mathrm{d}x)^2$ into the regularization parameter $\alpha$, \eqref{eq,contglobalMB} gives 
\begin{equation}\label{eq,discglobalMB}
\mathcal{M}(a)=\sum_{i=1}^{M^2}m(a(i)).
\end{equation}

Total variation has been widely studied and is well known to promote piecewise constant images with small perimeter\cite{TV,MB}. This combination, at least numerically \cite{Holrich}, allows us to significantly reduce the number of projections required to obtain a good reconstruction. For practical implementation we use a smoothed version of the isotropic total variation \cite{TV}
\begin{equation}\label{eq,isoTVMAT}
\mathrm{TV}_{c}(a)=\sum_{i=1}^{M^{2}-1}\sqrt{\|D_{i}a\|^{2}_{2}+c},		    
\end{equation}
where $c>0$ is a small smoothing constant and each $D_{i}\in\mathbb{R}^{2\times M^{2}}$ is a finite difference matrix as in \cite{Holrich}.
Note that the smoothness of the total variation is required in order to guarantee Lipschitz continuity of its gradient.
				    			    
After discretising, the variational problem \eqref{eq,ContVariationalprob} is equivalent to 

\begin{equation}\label{eq,discVariationalprob}
\mathrm{argmin}_{a,f\in \mathbb{R}^{M^2}} \mathcal{R}(a,f):=\|R[a]f-d\|^{2}_{2}+\alpha \mathcal{M}(a)+\lambda \mathrm{TV}(a)+\eta\mathrm{TV}(f)
\end{equation}
with $\mathrm{TV}$ and $\mathcal{M}$ as in \eqref{eq,isoTVMAT} and $\eqref{eq,discglobalMB}$.

Following \cite{Holrich} we use the alternating minimization scheme described in \cite{alternating}

\begin{equation}\label{eq,Alternating1}
\begin{split}
a^{k+1}&\in\mathrm{argmin}_{a}\mathcal{R}(a,f^{k})+\frac{1}{2\xi^{k}}\|a-a^{k}\|^{2},\\				    			     
f^{k+1}&\in\mathrm{argmin}_{f}\mathcal{R}(a^{k+1},f)+\frac{1}{2\xi^{k}}\|f-f^{k}\|^{2},				    			     
\end{split}
\end{equation}
for sufficiently small $\{\xi_{k}\}_{k=1}^\infty$. Then each alternating update is itself split into two updates and solved using an ADMM algorithm \cite{addm} as described in further detail in \cite{Holrich}. Note that we also use an Iterative Shrinkage Thresholding Algorithm (ISTA) in the ADMM algorithm when we update $a$. We also use the same adaptive scheme for $\beta$ given in \cite{Holrich}.

We then use following joint reconstruction algorithm from \cite{Holrich}.
\begin{algorithm}[H]
\caption{Joint reconstruction algorithm}\label{jointrecon}
\begin{algorithmic}[1]
\State Input $a^0$ as initial guess, step sizes $t,\beta^{0}$, tolerances $\delta_{1},\delta_{2},\delta_{3},\delta_4,\delta_5$ and regularization parameters $\alpha,\lambda$ and $\mu $.
\State Set $f^{0}$ to be the least squares solution of $\|R[a^{0}]f-d\|^{2}$.
\For{$k\geq 0$}
\State Set $x^0 = a^k$ and $y^0 = D x^0$.
\For{$l\geq 0$}
\State Update $x^{l+1}$ via ISTA or FISTA with $\delta_1$ as a tolerance on $\|x^{l+1}-x^{l}\|$.
\State Update $y^{l+1}$ via gradient descent.
\State Set $\mu^{l+1}=\mu^{l}+\beta^{l}(y^{l+1}-Dx^{l+1})$.
\State Update $\beta^{l+1}$ via adaptive scheme.
\State Terminate when $r^{l}<\delta_2$ and $s^{l}<\delta_{3}$ and output $a^{k+1} = x^{l+1}$.
\EndFor
\State Update $f^{k+1}$ using ADMM with tolerance $\delta_4$.
\State Terminate when $\|a^{k+1}-a^{k}\|_{2}<\delta_5$ and $\|f^{k+1}-f^{k}\|_{2}<\delta_5$.
\EndFor
\end{algorithmic}
\end{algorithm}		
\noindent As in \cite{Holrich}, in this algorithm $\beta^0$ is reset to the same initialised value whenever the inner iterations aimed at the $a$ update in \eqref{eq,Alternating1} (those indexed by $l$) restart. The next section gives some numerical reconstructions obtained from Algorithm \ref{jointrecon} for phantoms which may or may not satisfy Theorem \ref{thm1}.

\section{Numerical Examples}\label{sec,NumericalExamples}

In section \ref{subsec,jumps}, we showed that, provided certain limits do not vanish, we can recover jumps in $a$ even if $f$ jumps at the same place. If the relevant limit does vanish, we say it is a {\it cancelling case}. We now examine some cases where the methods in sections \ref{subsec,jumps} and \ref{subsec,boundarydetermine} do not apply by considering a radial and an edge cancelling case. We also give reconstructions in both cases for cancelling and non-cancelling examples. Our examples appear to show that reconstruction is still possible using the algorithm presented in section \ref{sec,NumericalMethod} even in cancelling cases.

\subsection{Radial cancelling case}\label{section,radialcase}

Consider the family of $a$ and $f$ indicated by Figure \ref{fig,circlegroundtruth}. This family consists of $a$ which are constant with value $c$ on a circular region of radius $0.5$ and equal to zero otherwise. The phantoms for $f$ in this family are composed of two regions: an inner disk of radius $0.5$ on which $f=f_2$ and an annulus of outer radius $0.8$ on which $f=f_1$ so that $f_2-f_1$ is the jump across the inner circle of radius $0.5$.
\begin{figure}
    \centering
    \includegraphics[scale=0.5]{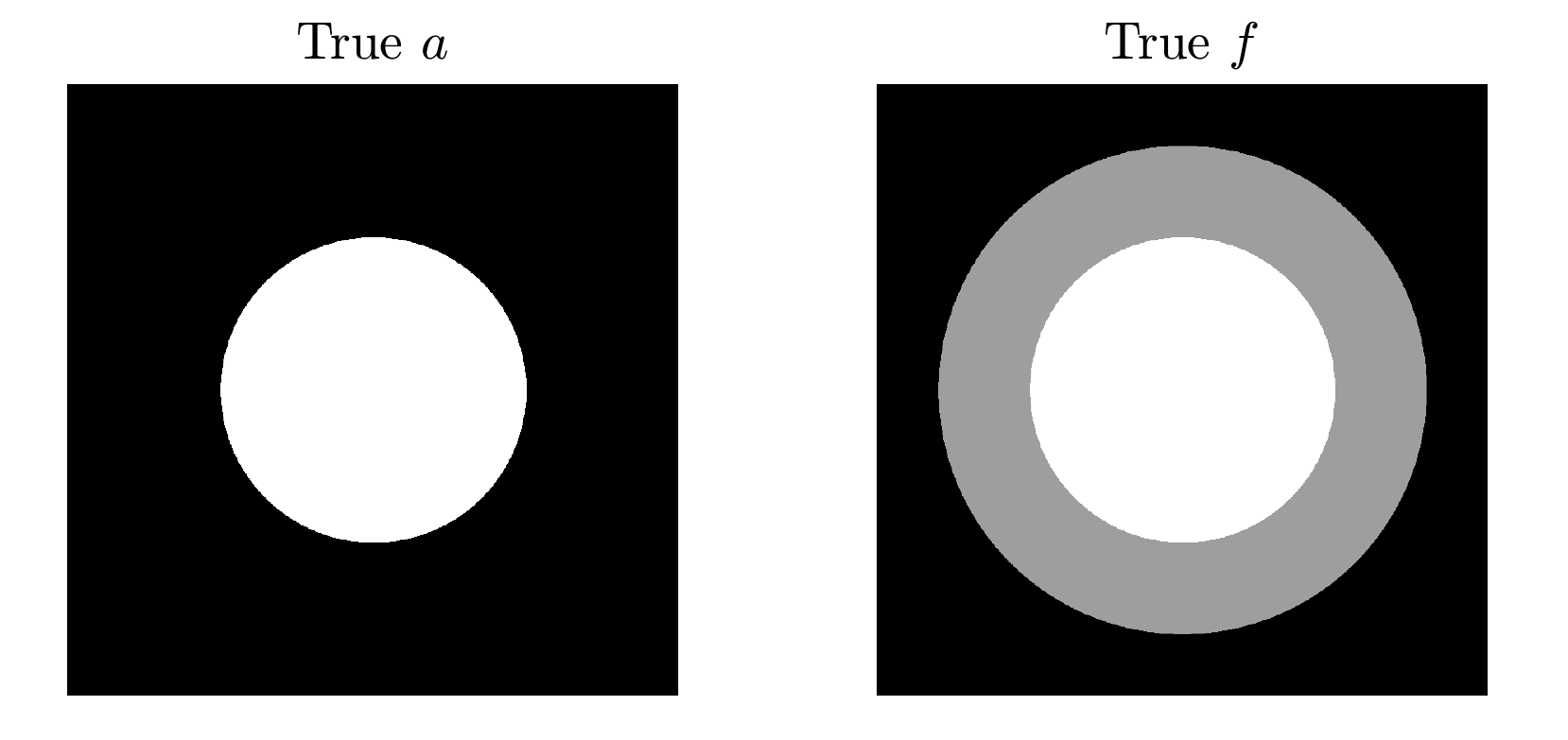}
    \caption{The left figure shows $a$ and the right shows $f$. Both $a$ and $f$ are multi-bang with $a$ taking two values either $0$ or $c$ and $f$ taking three values from $\{0,f_1,f_2 \}$, where $f_1$ is the value of $f$ on the annulus and $f_2$ is the value of $f$ on the inner circle and $f_1\neq f_2$. Here the inner radius is 0.5 and the outer radius is 0.8.}
    \label{fig,circlegroundtruth}
\end{figure}
In Lemma \ref{lem,tangentpoint}, we showed that the jump in $a$ can be recovered provided that the right hand side of equation \eqref{eq,dwRjump} is not zero. Given a point $\x \in\R^2$, the right hand side of \eqref{eq,dwRjump} can be computed exactly, for $a$ and $f$ as in Figure \ref{fig,circlegroundtruth}, in terms of $c,f_1$ and $f_2$. Note that in our case $f_{-}=f_1$, $f_{+}=f_2$, $c_+=c$ and $c_{-}=0$ for \eqref{eq,dwRjump}.

Let $\x$ be the most southerly point on a circle of radius 0.6 , so that $\x$ is in the annulus of $f$ and outside the support of $a$. We then consider the rays $L(\x,\theta(\omega))$ which pass through $\x$ and are tangent to the shared boundary of $a$ and $f$.  For a ray tangent to this shared boundary, setting \eqref{eq,dwRjump} to zero gives one equation for 3 unknowns $f_1$,$f_2$ and $c$. Since $f_1\neq f_2$ and $c>0$, each term on the right hand side of \eqref{eq,dwRjump} cannot be zero and therefore fixing 2 of the 3 variables $c,f_1,f_2$ gives exactly one value of the third which makes the right hand side of \eqref{eq,dwRjump} zero. For example, if we set $c=f_1=1$ and then plug these values into \eqref{eq,dwRjump} and set the right hand side to zero we obtain an equation involving known parameters and $f_2$ which we can solve. After some algebraic manipulation we find the triplet $f_1=1$, $a=1$ and $f_2\approx 1.6245$ causes the right hand side of \eqref{eq,dwRjump} to be zero and is therefore a cancelling case. 

\begin{figure}
    \centering
    \includegraphics[scale=0.5]{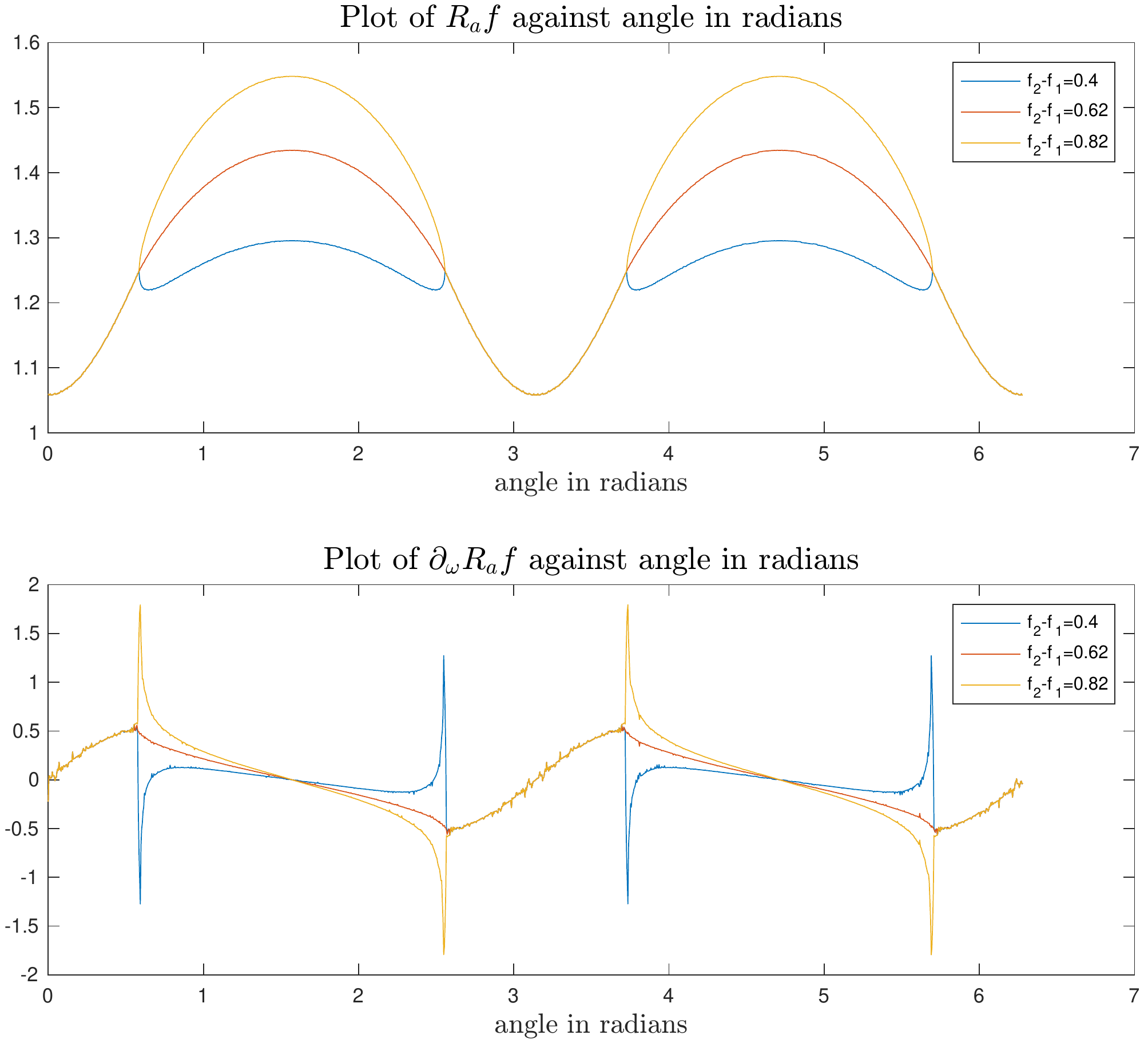}
    \caption{The top plot shows plots of $R_{a}f$ against angle $\omega$ and the bottom plot shows $\partial_{\omega}R_{a}f$ against angle. The blue, yellow and red lines correspond to the values of the jump $f_1$ which are larger, smaller and equal to the critical value respectively.}
    \label{fig,circlederplot}
\end{figure}
Figure \ref{fig,circlederplot} shows plots of the $R_{a}f$ and $\partial_{\omega}R_{a}f$ against angle $\omega$ for various choices of $f_2-f_1$ with $c=1$, and $f_1=1$. Note that, since $f_1=1$, for all cases the rays with angle $\omega$ between $2.65$ and $3.75$ do not intersect the inner circle of $a$ and $f$ and are therefore all the same. The red line corresponds with the critical choice of $f_1$ where the right hand side of \eqref{eq,dwRjump} is zero. We see that for both the yellow and blue lines there is a kink in the plot corresponding to $\omega$ when $L(\x,\theta(\omega))$ is tangent to the shared boundary whereas the red line for $f_2-f_1=0.6245$ is smooth. The bottom plot in Figure \ref{fig,circlederplot} shows the corresponding gradients. In both the yellow and blue lines we see that the gradient spikes at the tangent $\omega$ whereas the gradient of the red line is continuous for all $\omega$. This agrees with the theory as when the right hand side of \eqref{eq,dwRjump} is zero, $\partial_{\omega}R_{a}f$ is bounded. This also shows that the effects of cancelling cases are visible in the data $R_{a}f$.

\begin{figure}
\centering
\includegraphics[scale=0.4]{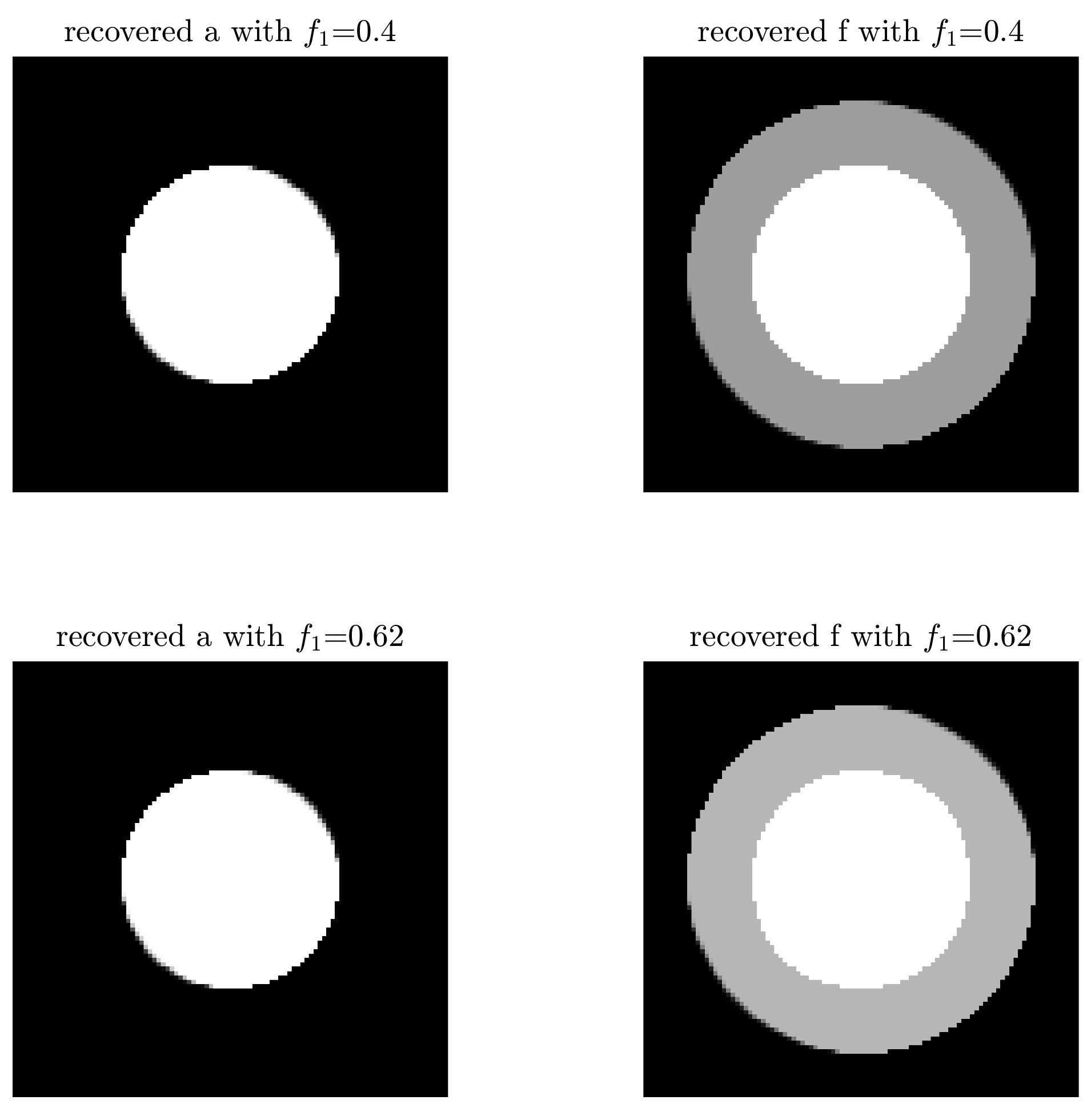}
\caption{Joint reconstructions of phantoms given in Figure \ref{fig,circlegroundtruth}. The top row uses $a$=1, $f_1=1$ and $f_2=1.4$. The bottom row uses $a=1$, $f_1=1$ and $f_2=1.62$}
\label{fig,circlerecovery}
\end{figure}
We now give some numerical reconstructions which show that, despite this effect in the data, we can still obtain usable reconstructions regardless of the choices of $f_1,f_2$ and $c$. Figure \ref{fig,circlerecovery} shows two joint reconstructions for the circle phantoms shown in Figure \ref{fig,circlegroundtruth}. The top row has $c=1$, $f_1=1$ and $f_2=1.4$ as the true phantom. The bottom row has $a=1$, $f_1=1$ and $f_2=1.62$ as the true phantom. The admissible set for $a$ in both cases is $\mathcal{A}=\{0,0.2,0.4,0.6,0.8,1\}$. Note that even though $f$ is multi-bang there is no multi-bang regularization on $f$. The choice of $a$,$f_1$ and $f_2$ in the top row is a non cancelling case whereas the choice in $a$, $f_1$ and $f_2$ in the bottom row is a cancelling case. In both cases, with suitable parameters, we are able to recover $a$ and $f$ with $5\%$ added Gaussian noise. Even when widely varying the reconstruction parameters $\alpha,\gamma_{a},\gamma_{f}$ and step sizes $t,\rho$ the reconstructed image remains the same for the bottom row. This means that despite the issues it causes our theoretical method, as well as the fact illustrated in Figure \ref{fig,circlederplot} that the singularity in $\partial_\omega R_{a}f$ does disappear, the reconstructions appear to be unique even in cancelling cases. Precisely what is causing this to happen is an area for further research. With an example of a potentially problematic case for a point of tangency considered, we now turn our attention to an interesting case for an edge.
 
\subsection{Edge cancelling case}

Throughout this section we are going to consider the family of phantoms indicated in Figure \ref{fig,squaregroundtruth}. 
\begin{figure}
    \centering
    \includegraphics[scale=0.5]{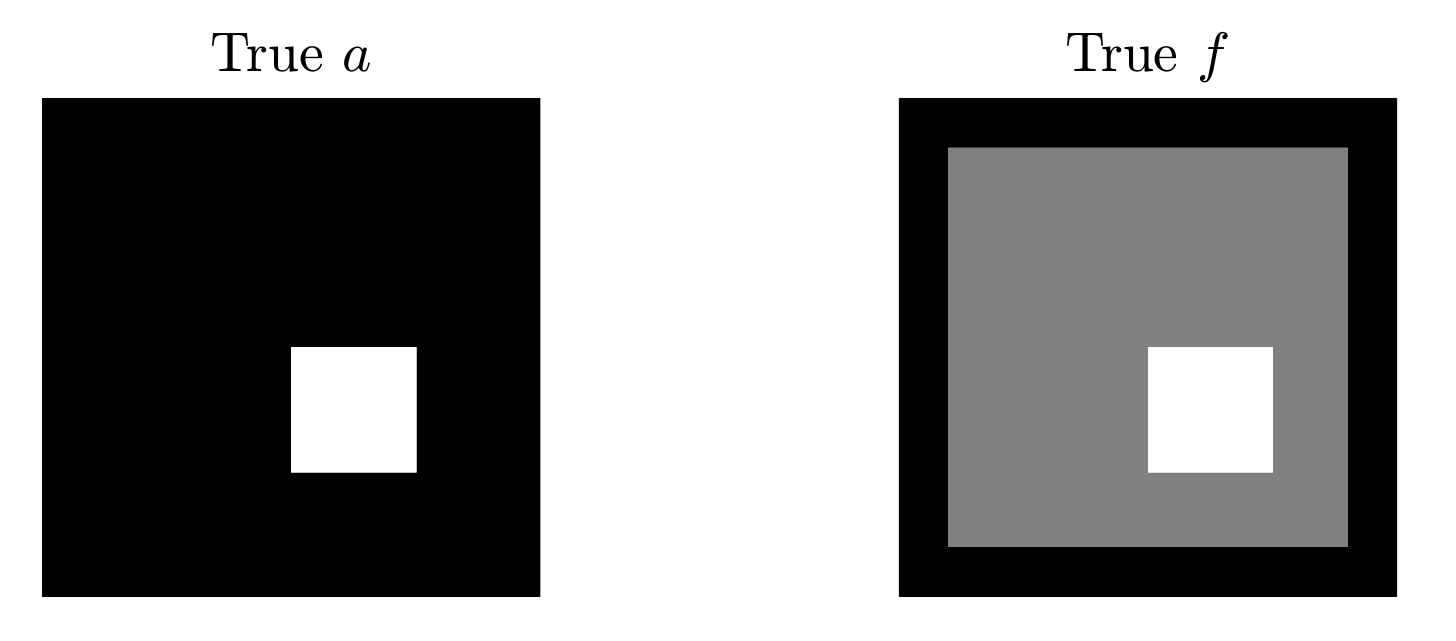}
    \caption{Phantoms for $a$ and $f$. Here, $a$ and $f$ are multi-bang with both $a$ and $f$ jumping over the inner square. The value of $f$ in the inner square is $f_2$ and in the outer region it is $f_1$.}
    \label{fig,squaregroundtruth}
\end{figure}
The left column in Figure \ref{fig,squaregroundtruth} is the phantom for $a$ which consists of a single square with $a=c$ on this square. The right column will be the phantom of $f$ which consists of two regions. The inner square, which is also present in the same location as the phantom for $a$, takes value $f_2$ and the outer region which takes value $f_1$. The relevant equation for this case is \eqref{eq,edge1}
Recalling the proof of Lemma \ref{lem,findboundary}, we can recover the jump in $a$ across an edge provided that $K_1(\x+\beta_-\theta^*, \theta^*) \neq 0$ is not zero where $\beta$ is the end point of the edge. Applying \eqref{eq,edge1}, the important quantity which cannot vanish is
\begin{equation}\label{eq,edgecancel}
	\begin{split}
	   &K_1(\x+ \beta_- \theta^*,\theta^*)= \\
	   & \hskip1cm 2 e^{-Da_-(\x+\beta_- \theta^*,\theta^*)} \Big ( \Delta_+ f(\x+\beta_- \theta^*,\theta^*)\\
	   & \hskip1cm-\frac{\Delta_+ a(\x+\beta_- \theta^*,\theta^*)}{2} \left (u(\x+\beta_- \theta^*,\theta(\omega^*_+) + u(\x+\beta_- \theta^*,\theta(\omega^*_-)) \right ) \Big ).
	   \end{split}
\end{equation}
In the following examples we fix $c=1$ and $f_2=1$ and then examine what happens for different choices of $f_1$. Instead of considering tangent rays to the boundary as in the radial case, we consider the bottom edge of the inner square shared by $a$ and $f$. In this case we have $f_{-}=f_1,f_{+}=1,a_{-}=0$ and $a_{+}=1$. In a similar manner to the radial case setting the right hand side of \eqref{eq,edgecancel} to zero gives an equation involving $f_1$ and known parameters which can be solved. After some algebraic manipulation, we find that the cancelling case for $f_2=c=1$ is when $f_1=0.4244.$

\begin{figure}[t]
    \centering
    \includegraphics[scale=0.6]{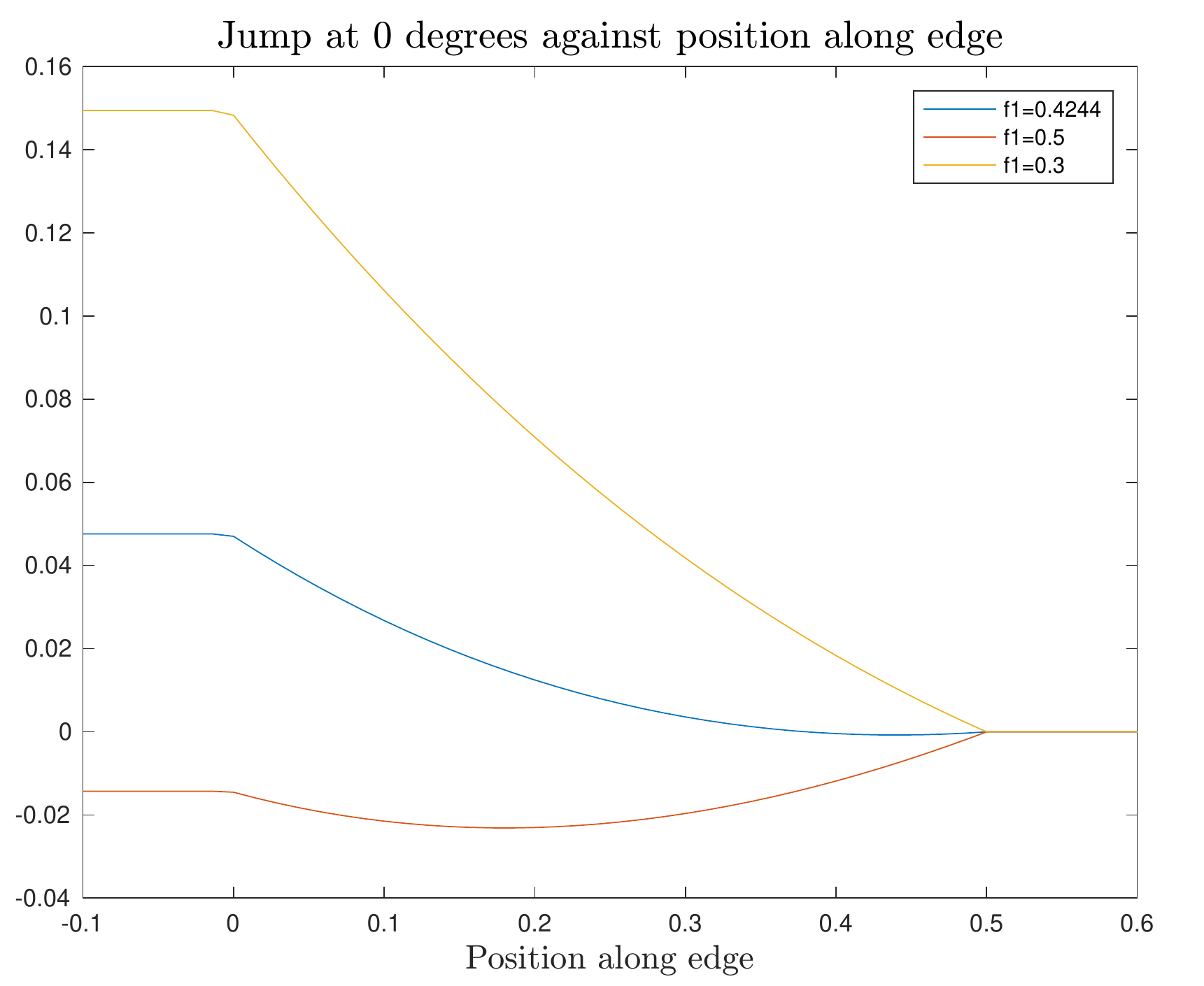}
    \caption{Plot of jump at $\omega=0$ against position along edge. Note that the edge lies between $0$ and $0.5$. The blue, red and yellow lines correspond to values of $f_1$ at, above and below the critical value 0.4244.}
    \label{fig,edgejumpplot}
\end{figure}
Figure \ref{fig,edgejumpplot} shows a plot of the jump in $R_{a}f(\x+s\theta^*,\theta(\omega))$ across $\omega=0$ as we travel along the straight line segment $E=\{\x=(x,y)\ : \ x\in[-0.1,0.6],~y=-0.5\}$. The bottom edge of the inner square, where the jump occurs, is the subset of $E$ with $x\in[0,0.5]$. In all three cases, before $x=0$ the entire edge is visible and therefore the jump is constant until we hit the start of the edge at $0$. Increasing the value of $f_1$ decreases the initial jump because $f_2$ is fixed at $1$ and is attenuated, whereas outside the inner square $a=0$. After $0.5$ the edge is passed and there is no jump regardless of the choice of $f_1$. The yellow line corresponds with a choice of $f_1$ smaller than the critical value and in this case we see that at the end of the edge there is a kink before going to zero (i.e. the derivative is not continuous). Similarly in the case where $f_1$ is larger than the critical value (the red graph) we see a kink in the graph again at $x=0.5$. When $f_1=0.4244$ the graph is smooth and has no kink at $0.5$ (i.e. the derivative does not jump). This matches the theory given in section \ref{sec,theory} as we would expect the cancelling case to go smoothly to zero as we approach the end of the edge. This also means that, as in the radial case, we are able to see the effects of a cancelling choice of $c,f_1$ and $f_2$ in the data itself.

As in the radial case given in section \ref{section,radialcase}, even though the effects of the choice of $f_1$ can be witnessed in the data, we can still obtain usable reconstructions from potentially problematic cases. Figure \ref{fig,squarerecovery} shows two reconstruction cases using the phantom for $a$ and $f$ given in Figure \ref{fig,squaregroundtruth}. In both cases the admissible set for $a$ is $\mathcal{A}=\{0,0.2,0.4,0.8,1 \}$. Similar to the radial case there do not appear to be any additional noticeable artefacts present in the $f_1=0.4244$ case. The only degradation appears to be in the top edge of the larger square. Again, even widely varying the reconstruction parameters, provided the algorithm converges, does not change the solution significantly which implies that there are no extra solutions introduced when we have cancellation in equation \eqref{eq,edgecancel}. In this case, it is possible that cancellation  does not occur when considering the line oriented in the opposite direction (i.e. \eqref{eq,edgecancel} is not zero at the opposite end of the edge for the line oriented in the opposite direction) meaning our theoretical results may still apply. This concludes the numerical results section. We now give a summary of the work in this paper as well as a few avenues for further research.
\begin{figure}[H]
    \centering
    \includegraphics[scale=0.6]{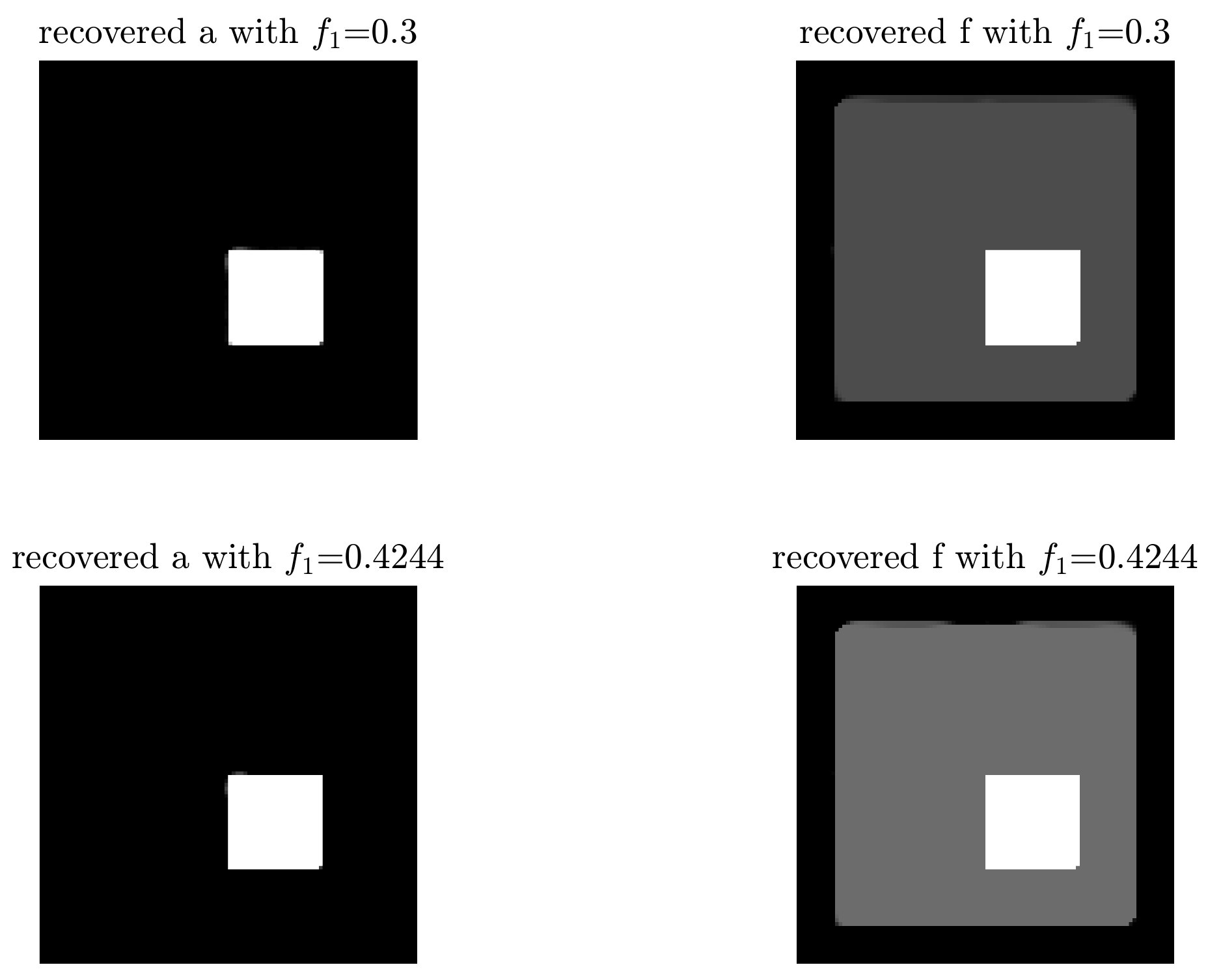}
    \caption{Joint reconstructions of $a$ and $f$. The left column shows $a$ and the right column shows $f$. The top row shows reconstructions for $a$ and $f$ when $f_1=0.3$ in the true phantom and the bottom row uses $f_1=0.4244$.}
    \label{fig,squarerecovery}
\end{figure}

\section{Conclusion}

In this paper we have solved the identification problem for SPECT with multi-bang $a$ and p.a.b. $C^2$ $f$ under non-cancelling conditions. 

We have established conditions which match up known microlocal analysis results for the linearized Attenuated X-Ray transform and been able to produce additional results for cases when $a$ and $f$ jump over straight edges.

We have also given some numerical reconstructions for cases where the theoretical methods presented in section \ref{sec,thmdef} are no longer valid. We have shown that even when we consider cancelling cases joint recovery of $a$ and $f$ is possible. The exact cause of the unique recovery is something which needs further research.

\bibliography{mbpaper2}{}
\bibliographystyle{plain}
\end{document}